\documentclass{amsart}
\usepackage[utf8]{inputenc}
\usepackage{amsmath, amssymb}
\usepackage{amsthm}
\usepackage[all]{xy}
\usepackage{color}
\definecolor{darkblue}{rgb}{0,0,0.6}
\usepackage[ocgcolorlinks,colorlinks=true, citecolor=darkblue, filecolor=darkblue,
linkcolor=darkblue, urlcolor=darkblue]{hyperref}
\usepackage{bbm}

\makeatletter
\@namedef{subjclassname@2010}{%
  \textup{2010} Mathematics Subject Classification}
\makeatother
\date{April 13, 2015}
\title{On the $K$-theory of subgroups of virtually connected Lie groups}
\author{Daniel Kasprowski}
\address{Max-Planck-Institut für Mathematik, Vivatsgasse 7, 53111 Bonn, Germany}
\email{kasprowski@mpim-bonn.mpg.de}

\newcommand{\bbC}{\mathbbm{C}}
\newcommand{\bbQ}{\mathbbm{Q}}

\newcommand{\bbF}{\mathbb{F}}
\newcommand{\bbR}{\mathbbm{R}}

\newcommand{\bbZ}{\mathbbm{Z}}
\newcommand{\bbN}{\mathbbm{N}}
\newcommand{\bbL}{\mathbb{L}}
\newcommand{\bbK}{\mathbb{K}}
\newcommand{\mcS}{\mathcal{S}}

\DeclareMathOperator{\Aut}{Aut}

\DeclareMathOperator{\ind}{ind}
\DeclareMathOperator{\Map}{Map}

\DeclareMathOperator*{\colim}{colim}
\newcommand{\VCyc}{\mathcal{V}cyc}

\newcommand{\mcF}{\mathcal{F}}
\newcommand{\mcA}{\mathcal{A}}

\newcommand{\mcB}{\mathcal{B}}

\newcommand{\mfh}{\mathfrak{h}}

\newcommand{\Fin}{{\mathcal{F}in}}

\newcommand{\comment}[1]{}
\newcommand{\co}{\colon\thinspace}

\newtheorem{thm}{Theorem}[section]
\newtheorem{prop}{Proposition}[section]
\newtheorem{cor}{Corollary}[section]
\newtheorem{lemma}{Lemma}[section]
\theoremstyle{definition}
\newtheorem{rem}{Remark}[section]

\makeatletter
\let\c@lemma=\c@thm
\let\c@prop=\c@thm
\let\c@cor=\c@thm
\let\c@example=\c@thm
\let\c@defi=\c@thm
\let\c@rem=\c@thm
\let\c@nota=\c@thm
\makeatother

 \newtheoremstyle{TheoremNum}
        {}{}              
        {\itshape}                      
        {}                              
        {\bfseries}                     
        {.}                             
        { }                             
        {\normalfont\bfseries\thmname{#1}\thmnote{#3}}
\theoremstyle{TheoremNum}
\newtheorem{clear}[]{}

\begin{document}

\begin{abstract}
We prove that for every finitely generated subgroup $G$ of a virtually connected Lie group which admits a finite-dimensional model for $\underbar EG$ the assembly map in algebraic $K$-theory is split injective. We also prove a similar statement for algebraic $L$-theory, which in particular implies the generalized integral Novikov conjecture for such groups.
\end{abstract}

\maketitle

\section{Introduction}
For every group $G$ and every ring $R$ there is a functor $\bbK_R\co OrG\to \mathfrak{Spectra}$ from the orbit category of $G$ to the category of spectra sending $G/H$ to (a spectrum weakly equivalent to) the $K$-theory spectrum $\bbK(R[H])$ for every subgroup $H\leq G$. By $K$-theory we will always mean non-connective $K$-theory as constructed by Pedersen and Weibel \cite{MR802790}. For any such functor $F\co OrG\to \mathfrak{Spectra}$ a $G$-homology theory $\bbF$ can be constructed via
\[\bbF(X):=\Map_G(\_,X_+)\wedge_{OrG}F,\]
see Davis and L\"uck \cite{davislueck}. We will denote its homotopy groups by $H_n^G(\_,F):=\pi_n\bbF(X)$. Let $\mcF$ be a family of subgroups of $G$. The $K$-theoretic assembly map for $\mcF$ is the map
\begin{equation*}
\alpha_\mcF\co H_n^G(E_\mcF G;\bbK_R)\to H^G_n(pt;\bbK_R)\cong K_n(R[G])\label{eq1}
\end{equation*}
induced by the map $E_\mcF G\to pt$. Here $E_\mcF G$ denote the classifying space for the family $\mcF$, see L\"uck \cite{MR1757730}. The assembly map is a helpful tool to related the $K$-theory of the group ring $R[G]$ to the $K$-theory of the group rings over $H\in\mcF$. The assembly map can more generally be defined for any small additive $G$-category instead of $R$, see Bartels and Reich \cite{coefficients}. In this article all additive categories will be small.\\
Analogously, for every additive $G$-category $\mcA$ with involution and every family of subgroups $\mcF$ we can define the $L$-theoretic assembly map
\[\alpha_\mcF\co H_n^G(E_\mcF G;\bbL^{\langle-\infty\rangle}_\mcA)\to H^G_n(pt;\bbL^{\langle-\infty\rangle}_\mcA).\]
The Farrell--Jones conjecture states that the assembly maps $\alpha_{\VCyc}$ for the family of virtually cyclic subgroups in $K$- and $L$-theory are isomorphisms for all additive $G$-categories $\mcA$ (with involution) and all $n\in\bbZ$. It was first formulated in \cite{MR1179537}. The Farrell--Jones conjecture has been proved for a large class of groups, for example hyperbolic and CAT(0)-groups, see Bartels and L\"uck \cite{MR2993750,MR2967054}, Bartels, L\"uck and Reich \mbox{\cite{MR2421141,MR2385666}} and Wegner \cite{MR2869063}, virtually solvable groups, see Wegner \cite{solvable}, and lattices in virtually connected Lie groups, see Bartels, Farrell and L\"uck \cite{MR3164984} and Kammeyer, L\"uck and R\"uping \cite{FJClattices}. The Farrell--Jones conjecture implies that the assembly maps $\alpha_{\Fin}$ for the family of finite subgroups are split injective, see Bartels \cite[Theorem~1.3]{MR2012963}. The rational split injectivity of the map $\alpha_{\Fin}$ in $L$-theory implies the Novikov conjecture. The integral split injectivity of $\alpha_{\Fin}$ is called the \emph{generalized integral Novikov conjecture}, for more details see \autoref{sec:l-theory}.
Kasparov proved the Novikov conjecture for all discrete subgroups of virtually connected Lie groups in \cite[Theorem 6.9]{MR918241}. More generally, the Novikov conjecture is true for groups which uniformly embed into a Hilbert space, see Skandalis, Tu and Yu \cite{MR1905840}. This includes all amenable groups and all groups with finite asymptotic dimension. By Carlsson and Goldfarb \cite[Section~3]{MR2058455} and Ji \cite[Corollary 3.4]{MR2144540} discrete subgroups of virtually connected Lie groups have finite asymptotic dimension giving a second prove that the Novikov conjecture holds for these groups. 
Here we will in particular show that discrete subgroups of virtually connected Lie groups also satisfy the generalized integral Novikov conjecture.

In \cite{KasFDC} the author proved the split injectivity of the assembly map for finitely generated subgroups $G$ of $GL_n(\bbC)$ which have an upper bound on the Hirsch length of the unipotent subgroups. For a definition of the Hirsch length see \autoref{sec:nil}. The bound on the Hirsch length exists if and only if $G$ has finite virtual cohomological dimension by Alperin and Shalen \cite{alperin}. Since $G$ is virtually torsion-free, this is the case if and only if there is a finite-dimensional model for $\underbar EG$ where we consider $G$ with the discrete topology, see L\"uck \cite[Theorem 3.1]{MR1757730}. In this article we want to extend this theorem to subgroups of all virtually connected Lie groups. Note that in the theorem we again consider $G$ with the discrete topology.
\begin{thm}
\label{thm:main}
Let $G$ be a finitely generated subgroup of a virtually connected Lie group and assume there exists a finite-dimensional model for $\underbar EG$. Then the $K$-theoretic assembly map
\[H_n^G(\underbar EG;\bbK_\mcA)\to K_n(\mcA[G])\]
is split injective for every additive $G$-category $\mcA$.
\end{thm}
A similar version holds for $L$-theory as well, which in particular implies the generalized integral Novikov conjecture for these groups, see \autoref{sec:l-theory}.

If $G$ is a discrete subgroup of a virtually connected Lie group $H$ and $K$ the maximal compact subgroup of $H$, then $H/K$ is a finite-dimensional model for $\underbar EG$, see L\"uck \cite[Theorem 4.4]{MR2195456}. In particular, we get the following corollary.

\begin{cor}
Let $G$ be a finitely generated discrete subgroup of a virtually connected Lie group. Then the $K$-theoretic assembly map
\[H_n^G(\underbar EG;\bbK_\mcA)\to K_n(\mcA[G])\]
is split injective for every additive $G$-category $\mcA$.
\end{cor}

The condition on the existence of a finite-dimensional model for $\underbar EG$ can be reformulated in the following way.
\begin{prop}
\label{prop:bound}
A finitely generated subgroup $G$ of a virtually connected Lie group admits a finite-dimensional model for $\underbar EG$ if and only if there exists $N\in\bbN$ such that every finitely generated abelian subgroup of $G$ has rank at most $N$.
\end{prop}
The rank of an abelian group $A$ is defined as $rk(A):=\dim_\bbQ(A\otimes_\bbZ \bbQ)$ or equivalently as the cardinality of maximal linearly independent subset of $A$. The statement that every finitely generated abelian subgroup of $G$ has rank at most $N$ is equivalent to the statement that every abelian subgroup of $G$ has rank at most $N$. For a proof of the proposition see \autoref{sec:nil}.\\

In \autoref{sec:inheritance} we prove that \autoref{thm:main} and its $L$-theoretic analog also hold without the assumption that $G$ is finitely generated.\\

\textbf{Acknowledgments:} I would like to thank Johannes Ebert for helpful discussions, and Henrik R\"uping and the referee for useful comments and suggestions. This work was partially supported by the SFB 878 ``Groups, Geometry and Actions'' and the Max Planck Society.

\section{Lie groups}
\label{sec:lie}
A Lie group is virtually connected if it has only finitely many connected components. For the rest of this section let $H$ be a virtually connected Lie group with Lie algebra $\mfh$ (which we identify with $T_eH$). The Lie group $H$ acts on itself by conjugation \[c\co H\to Aut(H),~g\mapsto(h\mapsto ghg^{-1}).\]
Taking the derivative yields a map
\[Ad\co H\to Aut(\mfh),~g\mapsto D_e(c(g)).\]
Since $Aut(\mfh)$ is a Lie subgroup of $GL(\mfh)$, $Ad$ gives a representation of $H$. The kernel of the representation $Ad$ is the centralizer $C_H(H_0)$ of the unit component $H_0$ of $H$.

By definition of the centralizer the group $C_H(H_0)\cap H_0$ is abelian and since $H$ is virtually connected the centralizer $C_H(H_0)$ is therefore virtually abelian. For every subgroup $G$ of $H$ we obtain a short exact sequence
\[1\to C_H(H_0)\cap G\to G\to Ad(G)\to 1.\]
with virtually abelian kernel and linear quotient. We will use this sequence to extend the results of \cite{KasFDC} to general virtually connected Lie groups. Before we can do so, we first need to prove \autoref{prop:bound} which will be done in the next chapter.
\section{A bound on the rank of abelian subgroups}
\label{sec:nil}
In the proof of \autoref{prop:bound} a bound on the Hirsch-length of the finitely generated nilpotent subgroups is needed. First we review some facts about nilpotent groups to see that this is the same as a bound on the rank of the finitely generated abelian subgroups.

Let $G$ be a group. Define $G_1:=G$ and recursively $G_{n+1}:=[G_n,G]$. The series $G=G_1\geq G_2\geq...$ is called the \emph{lower central series} of $G$. A group $G$ is \emph{nilpotent} if there exists $c\in\bbN$ with $G_{c+1}=1$. The smallest such $c$ is called the \emph{nilpotency class} of $G$, we denote it by $c(G)$. The \emph{upper central series} $1=Z_0(G)\leq Z_1(G)\leq\ldots$ of $G$ is recursively defined by \[Z_{i+1}(G):=\{g\in G\mid \forall h\in G: [g,h]\in Z_i(G)\}.\]
If $G$ is nilpotent then $Z_{c(G)}(G)=G$ and the length of the upper and lower central series agree. For any normal subgroup $H\leq G$ the quotient $G/H$ is again nilpotent.

The \emph{Hirsch-length} $h(G)$ of $G$ is
\[h(G):=rk(G_1/G_2)+...+rk(G_{c-1}/G_c)+rk(G_c),\]
where $rk(H)$ denotes the rank of an abelian group $H$, i.e. $rk(H):=\dim_\bbQ(H\otimes_\bbZ\bbQ)$.

Let $n(G)$ denote \[\max\{rk(A)\mid A\unlhd G\text{~an abelian normal subgroup}\}.\]

Let $H$ be a group and $G$ be a group acting on $H$. $G$ \emph{acts nilpotently} if there is a series
\[1=H_0\leq H_1\leq..\leq H_n=H\]
of $G$-invariant normal subgroups of $H$ such that the induced action on $H_i/H_{i-1}$ is trivial. In the special case where $H=G$ and the action is by conjugation $G$ acts nilpotently on itself if and only if $G$ is nilpotent. 
\begin{prop}
\label{prop:nilab}
Let $G$ be finitely generated nilpotent. Then $h(G)\leq \frac{n(G)(n(G)+1)}{2}$.
\end{prop}
The proposition is proved in M\"ohres {\cite[Theorem 2]{MR921696}} for torsion-free nilpotent groups instead of finitely generated nilpotent groups. For convenience of the reader we give a proof. For this we need the following well-known statements about nilpotent groups.
\begin{lemma}
\label{lem:fgsub}
A subgroup of a finitely generated nilpotent group is finitely generated.
\end{lemma}
\begin{proof}
The statement follows by induction on the nilpotency class.
\end{proof}
\begin{lemma}[{\cite[Theorem 1.3]{warfield}}]
Let $G$ be nilpotent and $N\unlhd G$ a non-trivial normal subgroup. Then $N\cap Z(G)$ is non-trivial, where $Z(G)$ denotes the center of $G$.
\end{lemma}
\begin{lemma}
Let $G$ be nilpotent and $A$ a maximal abelian normal subgroup. Then $C_G(A)=A$, where $C_G(A)$ is the centralizer of $A$ in $G$.
\end{lemma}
\begin{proof}
Since $A\unlhd G$ is normal, so is $C_G(A)$. Suppose $A\neq C_G(A)$. Then $C_G(A)/A$ is a non-trivial normal subgroup of $G/A$ and $H:=C_G(A)/A\cap Z(G/A)$ is non-trivial by the previous lemma. Let $C=\langle c\rangle$ be a cyclic subgroup of $H$. Then $C\unlhd Z(G/A)\unlhd G/A$ and since $C$ lies in the center it is a normal subgroup of $G/A$. Let $c'\in C_G(A)$ be a pre-image of $c$, then the pre-image of $C$ is $\langle A,c'\rangle$. This is abelian and normal in $G$, hence $A$ was not maximal with this property.
\end{proof}
\begin{lemma} Let $Tr(n,\bbZ)\leq GL_n(\bbZ)$ denote the subgroup of unitriangular matrices, i.e. every element of $Tr(n,\bbZ)$ has $1$'s on the diagonal and $0$'s below the diagonal. If $G\leq GL_n(\bbZ)$ acts nilpotently on $\bbZ^n$, then it is unipotent and conjugate to a subgroup of $Tr(n,\bbZ)$.
\end{lemma}
\begin{proof}
Since $Tr(n,\bbZ)$ is unipotent, it suffices to prove that $G$ is conjugate to a subgroup of it. Let
\[0=H_0\unlhd H_1\unlhd H_2\unlhd\ldots\unlhd H_k=\bbZ^n\]
be a sequence of $G$-invariant subspaces and let $G$ act trivially on $H_i/H_{i-1}$ for all $i=1,\ldots,k$. 
The lemma is obvious for $k=1$ and we will prove it by induction on $k$.
Let $H':=\{z\in\bbZ^n\mid \exists l\in \bbZ:lz\in H_1\}$. Let $z\in H',l\in\bbZ$ with $lz\in H_1$. For every $g\in G$ we have $lg(z)=g(lz)=lz$ and thus also $g(z)=z$, i.e. $G$ acts trivially on $H'$. By construction $\bbZ^n/H'$ is torsion-free and we obtain a splitting $\bbZ^n\cong H'\oplus \bbZ^n/H'$. The sequence
\[0=H'/H'\unlhd H_2+H'/H'\unlhd\ldots\unlhd H_k+H'/H'=\bbZ^n/H'\]  consists of $G$-invariant subspaces and $G$ acts trivially on the quotients. By induction there is a basis of $\bbZ^n/H$ such that $G\leq GL(\bbZ^n/H)$ is unitriangular. Using this basis together with a basis of $H'$ yields a basis of $\bbZ^n$ for which $G$ lies in $Tr(n,\bbZ)$.
\end{proof}
\begin{proof}[Proof of \autoref{prop:nilab}]
Let $n:=n(G)$ and $A$ be a maximal abelian normal subgroup. Then $A$ again is finitely generated by \autoref{lem:fgsub}\comment{cite[p. 163 Exercise 17.(a)]{bourbakialgebra}} and $A\cong \bbZ^n\oplus F$ with $F$ a finite group. $G$ acts by conjugation on $A$ and since $C_G(A)=A$, the induced map $G/A\to aut(A)$ is injective. Since $F$ is finite the projection to $aut(\bbZ^n)=GL_n(\bbZ)$ has finite kernel. The group $G$ is nilpotent and thus it acts nilpotently on $\bbZ^n$ (by conjugation). This implies that the image $G/A$ in $GL_n(\bbZ)$ is conjugate to a subgroup of the unitriangular matrices $Tr(n,\bbZ)$. Since $h(Tr(n,\bbZ))=\frac{n(n-1)}{2}$ we have 
\begin{eqnarray*}h(G)&\leq& h(A)+h(\ker(aut(A)\to GL_n(\bbZ)))+h(Tr(n,\bbZ)) \\ &=&n+0+\frac{n(n-1)}{2}=\frac{n(n+1)}{2}.\end{eqnarray*}
\end{proof}
A direct corollary of \autoref{prop:nilab} is the following.
\begin{cor}
\label{cor:nilab}
Let $G$ be a group. Then $G$ has a bound on the Hirsch-length of its finitely generated nilpotent subgroups if and only if it has a bound on the rank of its finitely generated abelian subgroups.
\end{cor}
Before we can prove \autoref{prop:bound} we need the following lemma.
\begin{lemma}
\label{lem:dimabelian}
Let $A$ be a (countable) abelian group with finite rank, then there is a finite-dimensional model for $\underbar EA$.
\end{lemma}
\begin{proof}
Let $rk A=n$. Then there exists a subgroup $B\leq A$ isomorphic to $\bbZ^n$. The quotient $Q:=A/B$ has rank $0$ and thus is a torsion group. For $n\in\bbN$ let $F_n\leq Q$ be finite subgroups with $F_n\leq F_{n+1}$ and $Q=\bigcup_{n\in\bbN}F_n$. Define a $Q$-CW complex $X$ by taking $\coprod_{n\in\bbN}Q/F_n$ as the zero skeleton and for every $n\in\bbN$ adding a 1-cell with stabilizer $F_{n}$ between the 0-cells $Q/F_n$ and $Q/F_{n+1}$. This defines a 1-dimensional model $X$ for $\underbar EQ$. Let $p\co A\to Q$ be the quotient map. For every finite subgroup $F\leq Q$ the preimage $p^{-1}(F)$ is finitely generated abelian of rank $n$ and thus has $\bbR^n$ as an $n$-dimensional model for $\underbar Ep^{-1}(Q)$. Therefore, the proof of L\"uck \cite[Theorem~3.1]{MR1757730} shows that $A$ has a model for $\underbar EA$ of dimension $n+1$. 
\end{proof}
Let $G$ be a subgroup of $GL_n(\bbC)$ and assume there exists $N\in\bbN$ such that the rank of every finitely generated unipotent subgroup of $G$ is at most $N$. Then by Alperin and Shalen \cite{alperin} the virtual cohomological dimension of $G$ is bounded and therefore admits a finite-dimensional model for $\underbar EG$ by \cite[Theorem 6.4]{MR1757730}. Using this we now can prove \autoref{prop:bound}.
\begin{proof}[Proof of \autoref{prop:bound}]
Let $G$ be a subgroup of a virtually connected Lie group $H$ such that there exists a finite dimensional model $X$ for $\underbar EG$. Then in particular $X$ is a model for $\underbar EA$ for every abelian subgroup $A\leq G$ and $rk A\leq \dim X$. 

For the other direction let $G$ be a finitely generated subgroup of a virtually connected Lie group $H$ such that there exists a bound on the rank of the finitely generated abelian subgroups of $G$. Then by \autoref{cor:nilab} $G$ has also a bound on the Hirsch-length of its finitely generated nilpotent subgroups. Let $G_0:=G\cap H_0$ and consider the extension 
\[1\to C_H(H_0)\cap G_0\to G_0\to Ad(G_0)\to 1\]
from \autoref{sec:lie}. Since $C_H(H_0)\cap G_0$ is contained in the center of $G_0$ also $Ad(G_0)$ has a bound on the the Hirsch-length of its finitely generated nilpotent subgroups and thus on the finitely generated unipotent subgroups. By the above it admits a finite dimensional model for $\underbar EAd(G_0)$.  And since also $K:=C_H(H_0)\cap G_0$ has finite rank, there is a finite dimensional model for $\underbar EK$ by \autoref{lem:dimabelian}. 
Consider the extensions
\[1\to K\to G_0\to Ad(G_0)\to 1\]
and
\[1\to G_0\to G\to F\to 1\]
with $F$ finite. The group $G_0$ is finitely generated since finite index subgroups of finitely generated groups are again finitely generated. Thus $Ad(G_0)$ is virtually torsion-free by Selberg's Lemma and we can use \cite[Theorem 3.1]{MR1757730} to obtain a finite dimensional model for $\underbar EG$ from these sequences.
\end{proof}

\begin{rem}
Using the results of the author from \cite{KasFDC} the short exact sequence \[1\to C_H(H_0)\cap G\to G\to Ad(G)\to 1\] implies that $G$ has fqFDC, which also is defined in \cite{KasFDC}. In particular, if $G$ has a bound on the order of the finite subgroups, then the main result of \cite{KasFDC} directly implies the split injectivity of the $K$-theoretic assembly map and a similar result in $L$-theory. Since we do not know if this always holds, we use a different approach using inheritance properties, see Sections \ref{sec:inh} and \ref{sec:proof}.
\end{rem}
\section{Inheritance properties}
\label{sec:inh}
To use the short exact sequence from \autoref{sec:lie} we want to show the following inheritance property.
\begin{prop}
\label{prop:inher} 
Assume there is a short exact sequence of groups
\[1\to J\to G\xrightarrow{\phi} Q\to 1\]
such that for every virtually cyclic subgroup $V\leq Q$ the pre-image $\phi^{-1}(V)$ satisfies the Farrell--Jones conjecture. Furthermore, assume that the assembly map
\[H_n^G(\underbar EQ;\bbK_\mcB)\to K_n(\mcB[Q])\]
is split injective for every $n\in\bbZ$ and every additive $Q$-category $\mcB$. Then the $K$-theoretic assembly map
\[H_n^G(\underbar EG;\bbK_\mcA)\to K_n(\mcA[G])\]
is split injective for every $n\in\bbZ$ and every additive $G$-category $\mcA$.
\end{prop}
\begin{proof}
Let $\mcA$ be an additive $G$-category. That $\phi^{-1}(V)$ satisfies the Farrell--Jones conjecture for every virtually cyclic subgroup $V\leq Q$ implies that the natural map $H^G_n(E_{\VCyc}G;\bbK_\mcA)\rightarrow H^G_n(E_{\phi^*\VCyc}G;\bbK_\mcA)$ is an isomorphism by Bartels and L\"uck \mbox{\cite[Lemma 2.2]{bartelsluecktrees}}, where $\phi^*\VCyc:=\{K\leq G\mid \phi(K)\in \VCyc\}$. Here we used that the projection $E_{\VCyc}G\times E_{\phi^*\VCyc}G\to E_{\phi^*\VCyc}G$ is a model for the natural map $E_{\VCyc}G\to E_{\phi^*\VCyc}G$. Furthermore, the natural map $H^G_n(\underbar EG;\bbK_\mcA)\to H^G_n(E_{\VCyc}G;\bbK_\mcA)$ is split injective by Bartels \cite{MR2012963}. Now the commutative diagram
\[\xymatrix{H^G_n(\underbar EG;\bbK_\mcA)\ar[r]\ar@{^(->}[d]&H^G_n(E_{\phi^*\Fin}G;\bbK_\mcA)\ar[d]\\H^G_n(E_{\VCyc}G;\bbK_\mcA)\ar[r]^\cong&H^G_n(E_{\phi^*\VCyc}G;\bbK_\mcA)}\]
implies that the map $H^G_n(\underbar EG;\bbK_\mcA)\to H^G_n(E_{\phi^*\Fin}G;\bbK_\mcA)$ is split injective, where $\phi^*\Fin:=\{K\leq G\mid \phi(K)\in\Fin\}$.
By Bartels and Reich \cite[Corollary 4.3]{coefficients} the split injectivity for $Q$ implies that the assembly map $H^G_n(E_{\phi^*\Fin}G;\bbK_\mcA)\to K_n(\mcA[G])$ is split injective. Combining these results yields the proposition.
\end{proof}
To apply the above proposition for the short exact sequence from the previous section we need the following.
\begin{lemma}
\label{lem:solv}
The class of virtually solvable groups is closed under group extensions.
\end{lemma}
The idea of the proof is taken from math.stackexchange.com, see \cite{437063}.
\begin{proof}
Let
\[1\to N\to G\xrightarrow{p} Q\to 1\]
be a short exact sequence and let $N,Q$ be virtually solvable. Let $Q'\leq Q$ be a solvable subgroup with $[Q:Q']<\infty$, then $[G:p^{-1}(Q')]<\infty$. Thus we can assume that $Q$ is solvable. We will first consider the case that $N$ is finite. Since $N$ is normal in $G$, $G$ acts on $N$ by conjugation, which induces a map $c\co G\to\Aut(N)$. The centralizer $C_G(N)$ of $N$ in the $G$ is the kernel of $c$. Since the class of solvable groups is closed under extension and $C_G(N)\cap N$ is abelian the exact sequence
\[1\to C_G(N)\cap N\to C_G(N)\to p(C_G(N))\to 1\]
shows that $C_G(N)$ is solvable.  The group $N$ is finite, thus $C_G(N)$ has finite index in $G$.

Now let $N$ be any virtually solvable group. And let $\mcS$ be the set of all normal, solvable, finite-index subgroups of $N$ ordered by inclusion. This is not empty and we can choose $K$ to be a maximal element of $S$. For every $g\in G$ also $gKg^{-1}K$ is a solvable, normal, finite-index subgroup of $N$. Since $K$ was maximal it therefore has to be normal in $G$. From the short exact sequence
\[1\to N/K\to G/K\to Q\to 1\]
it follows from the first case that $G/K$ is virtually solvable. Since $K$ is solvable the sequence
\[1\to K\to G\to G/K\to 1\]
implies that $G$ is virtually solvable.
\end{proof}
\section{Proof of \texorpdfstring{\autoref{thm:main}}{Theorem 1.1}}
\label{sec:proof}
For this section let $H$ be a virtually connected Lie group and $G\leq H$ a finitely generated subgroup such that there exists a finite dimensional model for $\underbar EG$. The proof of \autoref{thm:main} follows easily from the statements of the previous section.

\begin{proof}[Proof of \autoref{thm:main}]
Let $\Gamma:=Ad(G)$ be the image of $G$ under $Ad\co H\to GL(\mfh)$. Since $C_H(H_0)\cap G\cap H_0$ is contained in the center of $G$, the pre-image of any unipotent subgroup $U$ of $Ad(G\cap H_0)$ is a nilpotent subgroup of $G\cap H_0$. By \autoref{cor:nilab} and \autoref{prop:bound} there is a bound on the Hirsch-length of the nilpotent subgroups of $G\cap H_0$ and in particular there is a bound on the Hirsch-length of $U$. Since $G\cap H_0$ has finite index in $G$ this implies that there also is a bound on the Hirsch-length of the unipotent subgroups of $\Gamma$. Now we can apply
\begin{clear}[{\cite[\textbf{Corollary 3}]{KasFDC}}]
Let $F$ be a field of characteristic zero, $\Gamma$ a finitely generated subgroup of $GL_n(F)$ with a global upper bound on the Hirsch rank of its unipotent subgroups. Then the $K$-theoretic assembly map \[H_*^\Gamma(\underbar EG;\bbK_\mcA)\to H^\Gamma_*(pt;\bbK_\mcA)\cong K_*(\mcA[\Gamma])\] is split injective for every additive $\Gamma$-category $\mcA$.
\end{clear}
Note that \cite[Corollary 3]{KasFDC} if stated only for rings instead of additive $\Gamma$-categories, but by \cite[Theorem 8.1]{KasFDC} it is true for any additive $\Gamma$-category.

Furthermore, by Wegner \cite{solvable} every virtually solvable group satisfies the Farrell--Jones conjecture. Using this and \autoref{lem:solv} we see that the sequence
\[1\to C_H(H_0)\cap G\to G\to Ad(G)\to 1\]
satisfies the conditions of \autoref{prop:inher}.  Therefore, the assembly map
\[H^G_*(\underbar EG;\bbK_\mcA)\to K_*(\mcA[G])\]
 is split injective for every additive $G$-category $\mcA$.
\end{proof}
\section{L-theory}
\label{sec:l-theory}
Most of the statements from the previous sections also hold for $L$-theory. For the rest of the section let $G$ be a finitely generated subgroup of a virtually-connected Lie group $H$ with a finite dimensional model for $\underbar EG$ and let $Q$ be the image of $G$ under $Ad\co H\to GL(\mfh)$. Furthermore, let $\phi$ denote $Ad|_{G}$ and let $\mcA$ be an additive $G$-category with involution. As above we obtain the commutative diagram
\[\xymatrix{H^G_n(\underbar EG;\bbL^{ \langle-\infty\rangle}_\mcA)\ar[r]\ar[d]&H^G_n(E_{\phi^*\Fin}G;\bbL^{ \langle-\infty\rangle}_\mcA)\ar[d]\\H^G_n(E_{\VCyc}G;\bbL^{ \langle-\infty\rangle}_\mcA)\ar[r]^\cong&H^G_n(E_{\phi^*\VCyc}G;\bbL^{ \langle-\infty\rangle}_\mcA)}\]
and the lower horizontal map is still an isomorphism by Bartels and L\"uck \cite[Lemma~2.2]{bartelsluecktrees} and Wegner \cite{solvable}. But for the vertical map on the left to be injective we need that for every virtually cyclic subgroup $V\subseteq G$ there is an $i_0\in\bbN$ such that for every $i\geq i_0$ we have $K_{-i}(\mcA[V])=0$, see Bartels \cite{MR2012963}.
Then it remains to show that
\[H^G_n(E_{\phi^*\Fin}G;\bbL^{ \langle-\infty\rangle}_\mcA)\to L^{ \langle-\infty\rangle}_n(\mcA[G])\]
is split injective. By Bartels and Reich \cite[Proposition 4.2 and Corollary 4.3]{coefficients} this follows if
\[H^Q_n(\underbar EG;\bbL^{ \langle-\infty\rangle}_{\ind_\phi\mcA})\to L^{ \langle-\infty\rangle}_n((\ind_\phi\mcA)[Q])\]
is split injective. See \cite{coefficients} for the definition of $\ind_\phi\mcA$. To apply \cite[Theorem 9.1]{KasFDC} as above, we need the further assumption that for every finite subgroup $H\leq Q$ there is an $i_0\in\bbN$ such that for every $i\geq i_0$ we have
\[0=K_{-i}((\ind_\phi\mcA)[H])\cong K_{-i}(\mcA[\phi^{-1}(H)]).\]
Since $\phi^{-1}(H)$ is virtually abelian we obtain the following version of the main theorem for $L$-theory.
\begin{thm}
\label{thm:mainL}
Let $G$ be a finitely generated subgroup of a virtually connected Lie group and assume there exists an $N\in\bbN$ such that every finitely generated abelian subgroup of $G$ has at most rank $N$. Let $\mcA$ be an additive $G$-category with involution. Assume further that for every virtually abelian subgroup $H$ of $G$ there is an $i_0\in \bbN$ such that for every $i\geq i_0$ we have $K_{-i}(\mcA[H])=0$, then the $L$-theoretic assembly map
\[H_n^G(\underbar EG;\bbL^{ \langle-\infty\rangle}_\mcA)\to L^{ \langle-\infty\rangle}_n(\mcA([G])\]
is split injective.
\end{thm}
For torsion-free groups $G$ the \emph{integral Novikov conjecture} states that the assembly map
\[H_n^G(EG;\bbL^{\langle-\infty\rangle}_\bbZ)\to L_n^{\langle-\infty\rangle}(\bbZ[G])\]
is injective. It is known that the integral Novikov conjecture is false for groups containing torsion. Following Ji \cite{MR2379803} we say that $G$ satisfies the \emph{generalized integral Novikov conjecture} if the assembly maps
\[H_n^G(\underbar EG;\bbL^{\langle-\infty\rangle}_\bbZ)\to L_n^{\langle-\infty\rangle}(\bbZ[G]),\quad H_n^G(\underbar EG;\bbK_\bbZ)\to K_n(\bbZ[G])\]
are injective. By L\"uck and Reich \cite[Propostion 2.20]{MR2181833} the relative rational assembly map
\[H_n^G(EG;\bbL^{\langle-\infty\rangle}_\bbZ)\otimes_\bbZ\bbQ\to H_n^G(\underbar EG;\bbL^{\langle-\infty\rangle}_\bbZ)\otimes_\bbZ\bbQ\]
is injective. Observe that, since the $\bbZ/2$-Tate cohomology groups vanish rationally, there is no difference between the various decorations in L-theory because of the Rothenberg sequence. Therefore, by \cite[Proposition 1.53]{MR2181833} the injectivity of the rational assembly map
\[H_n^G(EG;\bbL^{\langle-\infty\rangle}_\bbZ)\otimes_\bbZ\bbQ\to L_n^{\langle-\infty\rangle}(\bbZ[G])\otimes_\bbZ\bbQ\]
implies the Novikov conjecture about the homotopy invariance of higher signatures. In particular, the generalized integral Novikov conjecture implies the (classical) Novikov conjecture.

We will show that $K_{-n}(\bbZ[G])=0$ for $n>1$ and any virtually abelian group $A$. Therefore, \autoref{thm:mainL} implies the generalized integral Novikov conjecture for the groups $G$ appearing in the theorem, i.e. we get the following corollary.
\begin{cor}
Let $G$ be a finitely generated subgroup of a virtually connected Lie group and assume there exists an $N\in\bbN$ such that every finitely generated abelian subgroup of $G$ has at most rank $N$. Then $G$ satisfies the generalized integral Novikov conjecture.
\end{cor}
By Farrell and Jones \cite[Theorem 2.1]{MR1340838} for every virtually cyclic group $V$ and $n>1$
\[K_{-n}(\bbZ[V])=0.\]
Let $G$ be a group and let $X$ be a finite $G$-CW-complex with virtually cyclic stabilizers. By induction on the dimension of $X$ we prove that
\[H_{-n}^G(X;\bbK_\bbZ)=0\]
for every $n>1$. For $\dim X=0$ we have
\[H_{-n}^G(X;\bbK_\bbZ)\cong \bigoplus_{x\in X} \bbK_{-n}(\bbZ[G_x])=0,\]
where the stabilizers $G_x$ are virtually cyclic by assumption.
Assume the above holds for $m$ and let $\dim X=m+1$. Then we have the long exact sequence
\[0=H_{-n}^G(X^{(m)};\bbK_\bbZ)\to H_{-n}^G(X;\bbK_\bbZ)\to H^G_{-n}(X,X^{(m)};\bbK_\bbZ)\]
and
\[H^G_{-n}(X,X^{(m)};\bbK_\bbZ)\cong \bigoplus_{c\in C_m} \bbK_{-n-m-1}(\bbZ[G_c])=0,\]
where $C_m$ denotes the set of $m$-cells of $X$ and $G_c$ the (virtually-cyclic) stabilizer of the cell $c$. 
Since every virtually abelian group $A$ satisfies the Farrell--Jones conjecture we have
\[K_{-n}(\bbZ[A])\cong H_{-n}^A(X;\bbK_\bbZ)\cong \colim_{K}H_{-n}^A(AK;\bbK_\bbZ)=0,\]
where $X$ is a $A$-CW-complex model for $E_{\VCyc}A$ and the colimit is taken over all finite subcomplexes $K\subseteq X$.
\section{Inheritance under colimits}
\label{sec:inheritance}
In this section we want to show that \autoref{thm:main} and \autoref{thm:mainL} hold without the assumption that $G$ is finitely generated. 

By Bartels, Echterhoff and L\"uck \cite[Lemma 2.4, Lemma 6.2]{inheritance} for every system $G_\alpha$ of finitely generated subgroups of $G$ such that $\colim_\alpha G_\alpha\cong G$ the assembly map
\[H^G_n(\underbar EG;\bbK_\mcA)\to K_n(\mcA[G])\]
is the colimit of the assembly maps
\[H^{G_\alpha}_n(\underbar EG_\alpha;\bbK_\mcA)\to K_n(\mcA[G_\alpha]),\]
for any additive $G$-category $\mcA$. The same statement holds in $L$-theory for any additive $G$-category with involution.  Note that the statement in \cite{inheritance} is formulated for rings with $G$-action instead of additive $G$-categories, but the statement for $G$-categories holds in the same way. Furthermore, a finite-dimensional model for $\underbar EG$ gives a finite-dimensional model for $\underbar EG_\alpha$ by restricting the action to $G_\alpha$. So taking the colimit over all finitely generated subgroups proves that injectivity holds without the assumption that $G$ is finitely generated. For the construction of a splitting we need to see that the splittings for the finitely generated subgroups are natural with respect to the structure maps of the colimit. In the proof of \autoref{thm:main} and \autoref{thm:mainL} the assumption that $G$ is finitely generated is only needed to apply \cite[Corollary 3]{KasFDC} respectively its L-theoretic analog. So it suffices to see that the splittings constructed in \cite{KasFDC} are natural with respect to the structure maps of the colimit. 

We will use the definitions of controlled categories and bounded mapping spaces from \cite[Section 5 and 7]{KasFDC}. In the following let $X$ denote a finite dimensional simplicial model for $\underbar EG$. By Bartels, Farrell, Jones and Reich \cite[Section 6]{MR2030590} the assembly map \[H^G_n(\underbar EG;\bbK_\mcA)\to K_n(\mcA[G])\] can be identified with the map \[\colim_{K\subseteq X~fin.}\pi_{n+1}(\bbK\mcA_G(GK)^\infty)^G\to\colim_{K\subseteq X~fin.}\pi_{n}(\bbK\mcA_G(GK)_0)^G.\]
Now consider the diagram
\[\xymatrix{
\colim\limits_{K\subseteq X~fin.}\pi_{n+1}(\bbK\mcA_G(GK)^\infty)^G\ar[d]^f\ar[r]& \colim\limits_{K\subseteq X~fin.}\pi_{n}(\bbK\mcA_G(GK)_0)^G\ar[d]^j\\
\colim\limits_{K\subseteq X~fin.}\pi_{n+1}\Map^{bd}_G(X,\bbK\mcA_G(GK)^\infty)\ar[r]_h& \colim\limits_{K\subseteq X~fin.}\pi_{n}\Map^{bd}_G(X,\bbK\mcA_G(GK)_0).}\]
By \cite[Remark 7.7]{KasFDC} the map $f$ is an isomorphism and the map $h$ is an isomorphism in the situation of \cite[Corollary 3]{KasFDC}. 

Let $\Gamma\to \Lambda$ be an injective group homomorphism. For every $\Lambda$-set $J$ and every subcomplex $K\subseteq X$ we can define a map
\[(\prod^{bd}_J\mcA_\Gamma(\Gamma K)^\infty)^\Gamma\to(\prod^{bd}_J\mcA_\Lambda(\Lambda K)^\infty)^\Lambda\]
as follows. A controlled modul $(M_j)\in(\prod^{bd}_J\mcA_\Gamma(\Gamma K))^\Gamma$ is send to $(M'_j)_j$ with $(M'_j)_{h',x,t}:=\bigoplus_{[h]\in \Lambda/\Gamma}(M_{h^{-1}j})_{h^{-1}h',h^{-1}x,t}$ and analogously on morphisms. This map is well defined since $(M_j)$ is $\Gamma$-invariant.
The above maps induce a map
\[\Map^{bd}_\Gamma(X,\bbK\mcA_\Gamma(\Gamma K))\to \Map^{bd}_\Lambda(X,\bbK\mcA_\Lambda(\Lambda K))\]
for every finite subcomplex $K\subseteq X$ and in the special case where $J=\{pt\}$ we obtain a map
\[(\bbK\mcA_\Gamma(\Gamma K)^\infty)^\Gamma\to(\bbK\mcA_\Lambda(\Lambda K)^\infty)^\Lambda.\]
The same maps can be constructed with $\mcA_\Gamma(\Gamma K)^\infty$ and $\mcA_\Lambda(\Lambda K)^\infty$ replaced by $\mcA_\Gamma(\Gamma K)_0$ and $\mcA_\Lambda(\Lambda K)_0$ respectively. So they induce maps from the above diagram for $\Gamma$ to the same diagram for $\Lambda$. 
We will omit the technical proofs that the maps of the diagram are natural with respect to these maps and that under the identification with the assembly map they correspond to the structure maps of the colimit from \cite{MR2030590}. This shows that the splitting $f^{-1}\circ h^{-1}\circ j$ is natural with respect to the structure maps of the colimit.

Now let us consider the $L$-theoretic version. For \cite[Remark 7.7]{KasFDC} it was used that the category $(\prod_{j\in J}\bbK\mcA_G(GK)^\infty)^G\simeq \prod_{[j]\in G\backslash J}\bbK\mcA_G^{G_j}(GK)^\infty$ is weakly equivalent to $(\bbK\prod_{j\in J}\mcA_G(GK)^\infty)^G\simeq \bbK\prod_{[j]\in G\backslash J}\mcA_G^{G_j}(GK)^\infty$ for every $G$-set $J$ with finite stabilizers and every finite subcomplex $K\subseteq X$, where $G_j$ is the stabilizer of $j\in J$.
Let $H\leq G$ be finite, then \[K_{n}(\mcA_G^{G_j}(G/H)^\infty)\cong \prod_{G_j\backslash G/H}K_{n}(\mcA_G^{G_j}(G_j/(G_j\cap H))^\infty)\cong \prod_{G_j\backslash G/H}K_{n-1}(\mcA[G_j\cap H]).\] If for each finite subgroup $H\leq G$ there exists $N\in \bbN$ such that for each $n\geq N$ the groups $K_{-n}\mcA[H]$ vanish, then by induction on the cells this implies that for every finite subcomplex $K\subseteq X$ there exists $N\in\bbN$ such that for $n\geq N$ the groups $K_{-n}(\mcA_G^{G_j}(GK)^\infty)$ vanish. Therefore, under this assumption $L$-theory commutes with the above product and we get that the map
\[\phi\co\Map^{bd}_G(X,\bbL\mcA_G(GK)^\infty)\to\Map_G(X,\bbL\mcA_G(GK)^\infty)\]
is an isomorphism and under the above assumption also
\[\psi\co(\bbL\mcA_G(GK)^\infty)^G\to\Map_G(X,\bbL\mcA_G(GK)^\infty)\]
is an isomorphism, see \cite[Section 9]{KasFDC}. Since $\psi$ factors over $\phi$ the map
\[(\bbL\mcA_G(GK)^\infty)^G\to\Map^{bd}_G(X,\bbL\mcA_G(GK)^\infty)\]
is an isomorphism as well. Therefore, we obtain the naturality of the splitting as in the case for $K$-theory.
\bibliographystyle{amsalpha}
\bibliography{fqFDC}
\end{document}